\documentclass[leqno]{siamltex}
\usepackage{amsmath}
\usepackage{amssymb}
\usepackage{graphicx}

\newcommand{\bn}{{\bf n}}
\newcommand{\bx}{{\bf x}}

\newcommand{\pT}{{\partial T}}

\def\bbQ{\mathbb{Q}}
\def\T{{\mathcal T}}

\def\l{{\langle}}
\def\r{{\rangle}}
\def\3bar{{|\hspace{-.02in}|\hspace{-.02in}|}}

\newtheorem{algorithm}{Weak Galerkin Algorithm}

\def\ad#1{\begin{aligned}#1\end{aligned}}  \def\b#1{\mathbf{#1}}
\def\a#1{\begin{align*}#1\end{align*}} \def\an#1{\begin{align}#1\end{align}}

\title{A stabilizer free weak Galerkin
  method for the Biharmonic Equation on Polytopal Meshes}

\author{Xiu Ye\thanks{Department of
Mathematics, University of Arkansas at Little Rock, Little Rock, AR
72204 (xxye@ualr.edu). This research was supported in part by
National Science Foundation Grant DMS-1620016.}
\and
Shangyou Zhang\thanks{Department of
Mathematical Sciences, University of Delaware, Newark, DE 19716 (szhang@udel.edu).}
}

\begin{document}

\maketitle

\begin{abstract}
A new stabilizer free weak Galerkin (WG) method is introduced and analyzed for the
 biharmonic equation.
Stabilizing/penalty terms  are often necessary in the finite element formulations
    with discontinuous approximations  to ensure the stability of the methods.
Removal of  stabilizers will simplify finite element formulations and
   reduce programming complexity.
This  stabilizer free WG method has an ultra simple formulation
    and can work on general partitions with polygons/polyhedra.
 Optimal order error estimates in a discrete $H^2$ for $k\ge 2$ and in $L^2$ norm
    for $k>2$ are established for the corresponding weak Galerkin finite element
   solutions.
  Numerical results are provided to  confirm the theories.
\end{abstract}


\pagestyle{myheadings}

\section{Introduction}

We consider the biharmonic equation of the form
\begin{eqnarray}
\Delta^2 u&=&f\quad \mbox{in}\;\Omega,\label{pde}\\
u&=&g\quad\mbox{on}\;\partial\Omega,\label{bc-d}\\
\frac{\partial u}{\partial
n}&=&\phi\quad\mbox{on}\;\partial\Omega,\label{bc-n}
\end{eqnarray}
where $\Omega$ is a bounded polytopal domain in $\mathbb{R}^d$.

\medskip
For the biharmonic problem (\ref{pde}) with Dirichlet and Neumann boundary conditions (\ref{bc-d}) and (\ref{bc-n}), the corresponding
weak form is given by seeking $u\in H^2(\Omega)$ satisfying
$u|_{\partial \Omega}=g$ and $\frac{\partial u}{\partial
n}|_{\partial \Omega}=\phi$ such that
\begin{equation}\label{wf}
(\Delta u, \Delta v) = (f, v)\qquad \forall v\in H_0^2(\Omega),
\end{equation}
where $H_0^2(\Omega)$ is the subspace of $H^2(\Omega)$ consisting of
functions with vanishing value and normal derivative on
$\partial\Omega$.

It is known that $H^2$-conforming
methods  require $C^1$-continuous piecewise polynomials
on a simplicial meshes, which imposes
difficulty in practical computation. Due to the complexity in the
construction of $C^1$-continuous elements, $H^2$-conforming finite
element methods are rarely used in practice for solving the
biharmonic equation.

As an alternative approach, \  nonconforming \ and discontinuous
finite element methods have been developed for solving the
biharmonic equation over the last several decades. The Morley
element \cite{morley} is a well-known example of nonconforming
element for the biharmonic equation by using piecewise quadratic
polynomials. The weak Galerkin finite element methods use discontinuous
  approximations on general polytopal meshes introduced first in \cite{wy}.
Many WG finite element methods have been developed for forth order problems
\cite{wg-bi1,wg-bi2,wg-bi3,wg-bi4,wg-bi5,wg-bi6,wg-bi7}.
These weak Galerkin finite element methods
   for (\ref{pde})-(\ref{bc-n}) have the following
  symmetric, positive definite and parameter independent formulation:
\begin{equation}\label{wf1}
(\Delta_w u_h, \Delta_w v)+s(u_h,v) = (f, v).
\end{equation}
 The stabilizer $s(\cdot,\cdot)$ in (\ref{wf1}) is necessary to guarantee the well posedness and the convergence of the methods.

The purpose of the work is to further simplify the WG formulation (\ref{wf1}) by removing the stabilizer to obtain an ultra simple formulation for the biharmonic equation:
\begin{equation}\label{wf2}
(\Delta_w u_h, \Delta_w v) = (f, v).
\end{equation}
We can obtain a stabilizer free WG method (\ref{wf2}) by appropriately designing the weak Laplacian $\Delta_w$. The idea is to raise the degree of polynomials used to compute weak Laplacian $\Delta_w$.  Using higher degree polynomials in computation of weak Laplacian will not change the size, neither the global sparsity of the stiffness matrix.

This new stabilizer free WG method for the forth order problem (\ref{bc-d})-(\ref{bc-n}) has an ultra simple symmetric positive definite formulation (\ref{wf2}) and can work on general polytopal meshes. For second order elliptic problems, stabilizer free WG methods have been studied in \cite{liu,wy, sfwg-soe}. However for forth order problems, to the best of our knowledge, this is the first finite element method without any stabilizers for totally discontinuous approximations. Optimal order
error estimates in a discrete $H^2$ norm is established for the
corresponding WG finite element solutions. Error estimates in the $L^2$ norm are also derived with a sub-optimal order of
convergence for the lowest order element and an optimal order of
convergence for all high order of elements. Numerical results are
presented to confirm the theory of convergence.

\section{Weak Galerkin Finite Element Methods}\label{Section:wg-fem}

Let ${\mathcal T}_h$ be a partition of the domain $\Omega$ consisting of
polygons in two dimension or polyhedra in three dimension satisfying a set of conditions defined in \cite{wy-mixed} and additional conditions specified in Lemma \ref{l-m1} and Lemma \ref{l-m2}.
Denote by ${\cal E}_h$ the set of all edges or flat faces in ${\cal
T}_h$, and let ${\cal E}_h^0={\cal E}_h\backslash\partial\Omega$ be
the set of all interior edges or flat faces.

For simplicity, we adopt the following notations,
\begin{eqnarray*}
(v,w)_{\T_h} &=& \sum_{T\in\T_h}(v,w)_T=\sum_{T\in\T_h}\int_T vw d\bx,\\
 \l v,w\r_{\partial\T_h}&=&\sum_{T\in\T_h} \l v,w\r_\pT=\sum_{T\in\T_h} \int_\pT vw ds.
\end{eqnarray*}
Let $P_k(K)$ consist all the polynomials degree less or equal to $k$ defined on $K$.

First we introduce a set of normal directions on ${\cal E}_h$ as follows
\begin{equation}\label{thanks.101}
{\cal D}_h = \{\bn_e: \mbox{ $\bn_e$ is unit and normal to $e$},\
e\in {\cal E}_h \}.
\end{equation}
Then, we can define a weak Galerkin finite element space $V_h$ for $k\ge 2$ as follows
\begin{equation}\label{vh}
V_h=\{v=\{v_0,v_b, v_{n}\bn_e\}:\ v_0\in P_{k}(T),  v_b\in
P_{k}(e),
 v_{n}\in P_{k-1}(e), e\subset\partial T\},
\end{equation}
where $v_n$ can be viewed as an approximation of $\nabla v_0\cdot\bn_e$.

Denote by $V_h^0$ a subspace
of $V_h$ with vanishing traces,
$$
V_h^0=\{v=\{v_0,v_b,v_{n}\bn_e\}\in V_h, {v_b}|_e=0,\ v_{n}\bn_e\cdot\bn|_e=0,\
e\subset\partial T\cap\partial\Omega\}.
$$

A weak Laplacian operator, denoted by $\Delta_{w}$,
is defined as the unique polynomial $\Delta_{w}v \in P_j(T)$ for $j>k$ that
satisfies the following equation
\begin{equation}\label{wl}
(\Delta_{w} v, \ \varphi)_T = ( v_0, \ \Delta\varphi)_T-\l v_b,\
\nabla\varphi\cdot\bn\r_\pT +\l v_n\bn_e\cdot\bn, \ \varphi\r_\pT,\quad
\forall \varphi\in P_j(T).
\end{equation}

Let $Q_0$, $Q_b$ and $Q_n$ be the locally defined $L^2$ projections onto $P_{k}(T)$, $P_{k}(e)$ and $P_{k-1}(e)$ accordingly on each element $T\in\T_h$ and $e\subset \pT$. For the true solution $u$ of (\ref{pde})-(\ref{bc-n}), we define $Q_hu$ as
$$
Q_h u = \{Q_0 u,Q_bu,Q_n(\nabla u\cdot\bn_e)\bn_e\}\in V_h.
$$

\begin{algorithm}
A numerical approximation for (\ref{pde})-(\ref{bc-n}) can be
obtained by seeking $u_h=\{u_0,\;u_b,\ u_{n}\bn_e\}\in V_h$
satisfying $u_b=Q_b g$ and $u_{n}\bn_e\cdot\bn=Q_{n}\phi$ on $\partial \Omega$
and the following equation:
\begin{equation}\label{wg}
(\Delta_w u_h,\ \Delta_w v)_{\T_h}=(f,\;v_0) \quad\forall\
v=\{v_0,\; v_b,\ v_{n}\bn_e\}\in V_h^0.
\end{equation}
\end{algorithm}

\begin{lemma}
Let $\phi\in H^2(\Omega)$, then on any $T\in\T_h$,
\begin{equation}\label{key}
\Delta_{w} \phi = \bbQ_h (\Delta \phi),
\end{equation}
where $\bbQ_h$ is a locally defined $L^2$ projections onto $P_{j}(T)$ on each element $T\in\T_h$.
\end{lemma}
\begin{proof}
It is not hard to see that for any $\tau\in P_j(T)$ we
have
\begin{eqnarray*}
(\Delta_{w} \phi,\ \tau)_T &=& (\phi,\ \Delta\tau)_T + \langle
(\nabla \phi\cdot\bn_e)\bn_e\cdot\bn, \;\tau\rangle_{\pT}-\langle  \phi,\ \nabla\tau\cdot\bn \rangle_{\pT}\\
&=&(\phi, \Delta\tau)_T + \langle \nabla \phi\cdot\bn,\ \tau\rangle_{\partial
T}-\langle \phi,\ \nabla\tau\cdot\bn \rangle_{\pT}\\
&=&(\Delta \phi,\ \tau)_T=(\bbQ_h\Delta \phi,\ \tau)_T,
\end{eqnarray*}
which implies
\begin{equation}\label{key-Laplacian}
\Delta_{w}  \phi = \bbQ_h (\Delta \phi).
\end{equation}
It completes the proof.
\end{proof}
\section{Well Posedness}

For any $v\in V_h+H^2(\Omega)$, let
\begin{equation}\label{3barnorm}
\3bar v\3bar^2=(\Delta_wv,\Delta_wv)_{\T_h}.
\end{equation}

We introduce a discrete $H^2$ norm as follows:
\begin{equation}\label{norm}
\|v\|_{2,h} = \left( \sum_{T\in\T_h}\left(\|\Delta v_0\|_T^2+h_T^{-3} \|  v_0-v_b\|^2_\pT+h_T^{-1}\|(\nabla v_0-v_n\bn_e)\cdot\bn\|^2_\pT\right) \right)^{\frac12}.
\end{equation}

For any function $\varphi\in H^1(T)$, the following trace
inequality holds true \cite{wy-mixed},
\begin{equation}\label{trace}
\|\varphi\|_{e}^2 \leq C \left( h_T^{-1} \|\varphi\|_T^2 + h_T
\|\nabla \varphi\|_{T}^2\right).
\end{equation}

The main goal of this section is to obtain  the equivalence of  the two norms $\|\cdot\|_{2,h}$ and $\3bar\cdot\3bar$. To do so,  we need the two following lemmas.

\begin{lemma}\label{l-m1} Let $T$ be a convex
  polygon/polyhedron of size $h_T$ with edges/faces $e_1$, $e_2$, \dots $e_n$.
  Let $\lambda_1\in P_1(T)$, $\lambda_1|{e_1}=0$ and
   $\max_T \lambda_1=1$.   Let $\lambda_i\in P_1(T)$, $i>1$,  $\lambda_i|{e_i}=0$ and
   $\lambda_i(\b m_1)=1$ where $\b m_1$ is the barycenter of $e_1$.
 For any $f\in P_k(e_1)$, there is a unique polynomial $q=\lambda_1\lambda_2^2\cdots\lambda_n^2 q_k$ for
  some $q_k\in P_k(T)$ such that
\an{\label{f11}   (q,p)_T&=0 \qquad \forall p\in P_{k-1}(T), \\
   \label{f12}   \l \nabla q\cdot \b n-f, p\r_{e_1} &=0\qquad \forall p\in P_{k}(e_1),\\
   \label{f13}   \|q\|_T &\le C h_T^{3/2} \|f\|_{e_1}, }
where $C$ depends on the minimum angle and the smallest ratio $h_{e_i}/h_T$, and $C$ is
   defined in \eqref{C1} below.
\end{lemma}

\begin{proof} We prove $q$ is uniquely defined by \eqref{f11}--\eqref{f12}.
Let $f=0$ in \eqref{f12}.  As $T$ is convex,   $\lambda_i>0$ in the interior of $e_1$ for all
  $i>1$.   Because of the positive weight,  the vanishing weighted $L^2(e_1)$ inner-product
  forces $\nabla q\cdot \b n=0$ on $e_1$:
\an{\label{g1} \l \nabla q\cdot \b n, p\r_{e_1}
    =- \frac 1{h_T} \l  \lambda_2^2\cdots\lambda_n^2 q_k, p \r_{e_1}=0
      \quad\forall p\in P_{k}(e_1). }
Thus the vanishing weighted $L^2(T)$ inner-product
  forces $q=0$ on $T$:
\an{\label{g2} (q,p)T
    = ( \lambda_1^2\cdots\lambda_n^2 q_{k-1}, p )_T=0
      \quad\forall p\in P_{k-1}(T), }
where $q_{k-1}\lambda_1=q_k$.

We find some upper bounds and lower bounds of these weight functions $\lambda_i$.

Let $e_i$, $1<i\le m$ ($m=3$ in 2D), be a neighboring edge/face of $e_1$.
Using the distance as its variable,  we have
\a{ \lambda_i|_{e_1} = \frac 2{h_{e_1}}x, }
where $x$ is the distance, along $e_1$, of the point to the intersection of $e_1$ and $e_i$.
Here in 3D,  we assume the size of $e_1$ is roughly twice the distance from the
  barycenter $\b m_1$ to the intersection edge $e_1\cap e_i$.
To avoid too many constants,  we simply assume reasonably $h_{e_i}\ge h_T/4$.
We compute the maximum as
\an{\label{l1} \max_T \lambda_i = \frac{h_{\perp e_i}(T) }{ (h_{e_1}/2)\sin \alpha_i}
      \le \frac{2 h_T }{  h_{e_1} \sin \alpha_i}
      \le \frac 8 {   \sin \alpha_i}\le \frac 8 {   \sin \alpha_0}, }
where $\pi-\alpha_i$ ($\alpha_i\ge \alpha_0>0$, and $\alpha_i\le \pi-\alpha_0$)
   is the angle between $e_1$ and $e_i$,  and $h_{\perp e_i}(T)$ is the
  maximal distance of a point on $T$ to $e_i$.
For a lower bound,  we have
\an{ \label{l12} \lambda_i|_{T_0}
        & \ge \begin{cases}  \frac {15}{16}  & \hbox{if } \alpha_i \le \pi/2,\\
            1-\frac{\sqrt d}{16 \sin \alpha_i}
          \ge \frac12  & \hbox{if } \alpha_i > \pi/2, \end{cases}
   } where $T_0$ is a square/cube at middle of $e_1$ with size $h_{e_1}/16$, cf. Figure \ref{s-ball}.
We note that other than triangles, $\alpha_i\le \pi/2$ for most other polygons.
Here in \eqref{l12},  we assumed $\sin\alpha_0 \ge \sqrt{d}/8$, where $d$ is the space
  dimension, 2 or 3.

\begin{figure}[h!]
 \begin{center} \setlength\unitlength{2pt}
\begin{picture}(120,50)(0,0)
 \put(0,0){\line(1,0){80}}\put(15,2){$e_1$}\put(15,35){$T$}
 \put(55,0){\line(0,1){10}}\put(58,4){$T_0$}\put(60,0){\circle*{2}}
 \put(75,12){\vector(-1,-1){12}} \put(76,13){$e_0$} \put(41,13){$\b m_1$}
 \put(55,10){\line(1,0){10}}               \put(48,12){\vector(1,-1){11}}
 \put(65,10){\line(0,-1){10}}
 \put(80,0){\line(1,0){40}}
 \put(120,0){\line(0,1){50}}\put(112,28){$e_2$}
 \put(0,50){\line(1,0){120}}\put(48,46){$e_4$}
 \put(0,0){\line(0,1){50}}\put(2,26){$e_3$}
 \end{picture}
\label{s-ball}
\caption{Size $h_{e_0}=h_{e_1}/16$, and $T_0$ is square/cube of size $h_{e_0}$ at $\b m_1$. }
\end{center}
\end{figure}
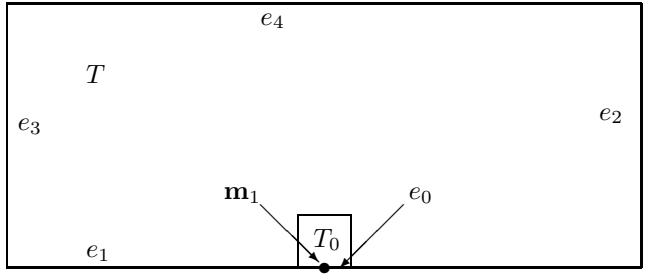

For non-neighboring edges $e_j$,  we have
\a{ \lambda_j|_{e_1} = \begin{cases} 1 & \hbox{if } e_j \parallel e_1,\\
           \frac {2(x + x_j)}{h_{e_1}+x_j}  & \hbox{otherwise, }\end{cases}
     }
where $x$ is the arc-length parametrization on $e_1$ toward the extended intersection of
  $e_1$ and $e_i$,  $x_j$ is the distance on $e_1$ from the an boundary point of $e_1$ to
   the intersection.
Supposing $e_i$ is the only edge/polygonal between $e_1$
  and $e_j$,  $x_j=h_{e_i}(\cos \alpha_i-\cos(\alpha_i+\alpha_j))$.
Because $x_j\ge 0$,  it follows that
\an{\label{l21} \max_T \lambda_j = \frac{h_{\perp e_j}(T) }{ (h_{e_1}/2)\sin \alpha_i}
      \le \frac{2 h_T }{ ( h_{e_1}+x_j) \sin (\alpha_i+\alpha_j)}
      \le \frac 8 {   \sin \alpha_0}. }
For a lower bound, because $x_j>0$  and $e_i$ is an edge/polygon in between,  we have
\an{\label{l22}  \lambda_j|_{T_0}
        & \ge  \lambda_i|_{T_0} \ge \frac 12.  }
Together, in \eqref{g1} and \eqref{g2},  we have, noting $\lambda_1|_T\le 1$,
\an{\label{bound} \lambda_2^2\cdots\lambda_n^2|_{T_0} \ge \frac{1} {2^{2n-2}},
     \quad\hbox{ and } \
   \lambda_1^2 \lambda_2^2\cdots\lambda_n^2|_T \le \frac{8^{2n-2}} {\sin^{2n-2} \alpha_0}.  }

Let  $\tilde q_k \in P_k(e_1)$ be the unique solution in  \eqref{f12},
   i.e., $\tilde q_k=q_k|_{e_1}$.
Letting $p=-h_T \tilde q_k$ in  \eqref{f12}, cf. \eqref{g1},   we get, by  \eqref{bound},
\a{  \frac{1} {16^{2k}} \frac{1} {2^{2n-2}}\|\tilde q_k\|_{e_1}^2
    & \le \frac{1} {2^{2n-2}}\|\tilde q_k\|_{e_0}^2
      \le  \| \lambda_2 \cdots\lambda_n \tilde q_k\|_{e_0}^2\\
     &\le  \| \lambda_2 \cdots\lambda_n \tilde q_k\|_{e_1}^2
        = \l  \lambda_2^2\cdots\lambda_n^2 \tilde q_k, \tilde q_k \r_{e_1}  \\
     &=-\frac{1}{h_T} \l \lambda_2^2\cdots\lambda_n^2 q_k , p \r_{e_1}
      =  \l f , p \r_{e_1}
       = \l f, -h_T \tilde q_k \r_{e_1} \\
    &\le h_T \|f\|_{e_1}\|\tilde q_k\|_{e_1}, }
where in the first step we use the fact $\tilde q_k$ is a degree $k$ polynomial.
For the unique solution $\tilde q_k \in P_k(e_1)$, we
   view it as a polynomial on the whole line or whole plane containing $e_1$.
We also extend it to $P_k(\mathbb{R}^d)$ by letting it be constant in the direction
  orthogonal to $e_1$.
Let $S_T$ be a square/cube of size $h_T$ containing $T$, with one side $S_{e_1}$ which contains
   $e_1$.
It follows that, by \eqref{bound},
\a{ \|\tilde q_k\|_T^2 & \le \|\tilde q_k\|_{S_T}^2=h_T \|\tilde q_k\|_{S_{e_1}}^2
    \le h_T (\frac {h_T}{h_{e_1}})^{2k} \|\tilde q_k\|_{ {e_1}}^2 \\
     & \le h_T 4^{2k}   \|\tilde q_k\|_{ {e_1}}^2
   \le h_T 4^{2k} ( 2^{8k+2n-2}   h_T \|f\|_{e_1})^2 \\
     & = 2^{20k +4n-4} h_T^3  \|f\|_{e_1}^2.
} We rewrite $q$ in terms of this extended polynomial,
\a{ q=\lambda_1\lambda_2^2\cdots\lambda_n^2 q_k
     =\lambda_1\lambda_2^2\cdots\lambda_n^2 (\lambda_1 q_{k-1} + \tilde q_k)
    \quad\hbox{for some } \ q_{k-1}\in P_{k-1}(T).
  } Letting $p=q_{k-1}$ in \eqref{f11},  it follows that, by \eqref{bound},
\a{ \|q_{k-1}\|_T^2 &\le (h_T/h_{e_0})^{2k-2} \|q_{k-1}\|_{T_0} ^2
          \le 64^{2k-2} 2^{2n-2} ( \lambda_2^2\cdots\lambda_n^2  q_{k-1},q_{k-1})_{T_0}\\
    &\le 2^{2n+12k-14} (2h_T/h_{e_0})^{ 2}( \lambda_1^2\cdots\lambda_n^2  q_{k-1},q_{k-1})_{T_{0,0}}
    \\ & \le 2^{2n+12k}(2h_T/h_{e_0})^{ 2n+2k-2} ( \lambda_1^2\cdots\lambda_n^2  q_{k-1},q_{k-1})_{T}
    \\&   =  2^{16n+26k-14} ( \lambda_1\lambda_2^2\cdots\lambda_n^2 \tilde q_{k},-q_{k-1})_{T}\\
    &\le  2^{16n+26k-14} \frac{8^{2n-2}} {\sin^{2n-2} \alpha_0}
    \| \tilde q_{k}\|_{T}  \|q_{k-1}\|_{T}, }
where $T_{0,0}$ is the top half of $T_0$, and we used the fact
   $\max_T \lambda_1=1$ and the fact that the integrant
   on $T_{0,0}$ is a degree $2n+2k-2$ polynomial.
We estimate
\a{ \|q\|_T^2 &= (\lambda_1^2\lambda_2^4\cdots\lambda_n^4 (\lambda_1 q_{k-1} + \tilde q_k),
   \lambda_1 q_{k-1} + \tilde q_k)_T
            \\& \le  \frac{8^{4n-4}} {\sin^{4n-4} \alpha_0} (\lambda_1 q_{k-1} + \tilde q_k,
   \lambda_1 q_{k-1} + \tilde q_k)_T
        \\& \le  \frac{8^{4n-4}} {\sin^{4n-4} \alpha_0} 2
          ( \|\lambda_1 q_{k-1}\|_T^2 + \| \tilde q_k\|_T^2 )
      \\ & \le  \frac{2^{12n-11}} {\sin^{4n-4} \alpha_0}
          ( \| q_{k-1}\|_T^2 + \| \tilde q_k\|_T^2 ). }
Combining above three bounds,  we get
\an{\label{C1}\ad{ \|q\|_T &\le
    \Big(\frac{2^{12n-11}} {\sin^{4n-4} \alpha_0}\Big)^{\frac 12}
       \Big( ( \frac{2^{22n+26k-20}} {\sin^{2n-2} \alpha_0})^2 + 1\Big)^{\frac 12}   \| \tilde q_k\|_T
            \\& \le
      \Big(\frac{2^{16n+20k-15}} {\sin^{4n-4} \alpha_0}\Big)^{\frac 12}
       \Big( ( \frac{2^{22n+26k-20}} {\sin^{2n-2} \alpha_0})^2 + 1\Big)^{\frac 12} h_T^{3/2} \|f\|_{e_1}
    =: C h_T^{3/2} \|f\|_{e_1}.
} } The proof is completed.
\end{proof}

\begin{lemma} \label{l-m2}
Let the following notations be defined in Lemma \ref{l-m1}.
 For any $g\in P_{k-1}(e_1)$, there is a unique polynomial $q=\lambda_2^2\cdots\lambda_n^2 q_{k+1}$ for
  some $q_{k+1}\in P_{k+1}(T)$ such that
\an{\label{f21}   (q,p)_T&=0 \qquad \forall p\in P_{k-1}(T), \\
   \label{f22}   \l \nabla q\cdot \b n , p\r_{e_1} &=0\qquad \forall p\in P_{k}(e_1),\\
   \label{f2-3}   \l   q-g, p\r_{e_1} &=0\qquad \forall p\in P_{k+1}(e_1),\\
   \label{f23}   \|q\|_T &\le C h_T^{1/2} \|g\|_{e_1}, }
where $C$ depends on the minimum angle and the smallest ratio $h_{e_i}/h_T$, and $C$ is
   defined in \eqref{C2} below.
\end{lemma}

\begin{proof}  For unisovence, letting $g=0$ in \eqref{f2-3},  we get
   $q=\lambda_1\lambda_2^2\cdots\lambda_n^2 q_{k}$ for some $q_k\in P_k(e_1)$,
   because the weights $\lambda_i>0$.
   By \eqref{f22}, $\nabla q\cdot \b n|_{e_1}=0$ and thus
   $q=\lambda_1^2\lambda_2^2\cdots\lambda_n^2 q_{k-1}$ for some $q_{k-1}\in P_{k-1}(T)$.
   By \eqref{f21}, $ q_{k-1}|_{T}=0$ and thus
   $q=0$.

The upper and lower bounds for $\lambda_i$ are same as that in Lemma \ref{l-m1}.

Let $\tilde q_{k+1}\in P_{k+1}(e_1)$ be the unique solution in \eqref{f2-3},
  i.e.,  $q|_{e_1}=\lambda_2^2\cdots\lambda_n^2 \tilde q_{k+1}$.
Letting $p=\tilde q_{k+1}$ in \eqref{f2-3}, we get, by \eqref{bound},
\a{   \frac{1} {16^{2k+2}} \frac{1} {2^{2n-2}}\|\tilde q_{k+1}\|_{e_1}^2
    & \le \frac{1} {2^{2n-2}}\|\tilde q_{k+1}\|_{e_0}^2
      \le  \| \lambda_2 \cdots\lambda_n \tilde q_{k+1}\|_{e_0}^2\\
     &\le  \| \lambda_2 \cdots\lambda_n \tilde q_{k+1}\|_{e_1}^2
        = \l  \lambda_2^2\cdots\lambda_n^2 \tilde q_{k+1}, \tilde q_{k+1} \r_{e_1}  \\
     & = \l g,  \tilde q_{k+1} \r_{e_1}  \le   \|g\|_{e_1}\|\tilde q_{k+1}\|_{e_1}, }
where in the first step we use the fact $\tilde q_{k+1}$ is a degree ${k+1}$ polynomial.
For the unique solution $\tilde q_{k+1} \in P_{k+1}(e_1)$, we
   view it as a polynomial on the whole line or whole plane containing $e_1$.
We also extend it to $P_{k+1}(\mathbb{R}^d)$ by letting it be constant in the direction
  orthogonal to $e_1$.
Let $S_T$ be a square/cube of size $h_T$ containing $T$, with one side $S_{e_1}$ which contains
   $e_1$.
It follows that, by \eqref{bound},
\a{ \|\tilde q_{k+1}\|_T^2 & \le \|\tilde q_{k+1}\|_{S_T}^2=h_T \|\tilde q_{k+1}\|_{S_{e_1}}^2
    \le h_T (\frac {h_T}{h_{e_1}})^{2k+2} \|\tilde q_{k+1}\|_{ {e_1}}^2 \\
     & \le h_T 4^{2k+2}   \|\tilde q_{k+1}\|_{ {e_1}}^2
   \le h_T 4^{2k+2} ( 2^{8k+2n+6}    \|g\|_{e_1})^2 \\
     & = 2^{20k +4n+16} h_T  \|g\|_{e_1}^2.
} We rewrite $q$ in terms of this extended polynomial,
\a{ q= \lambda_2^2\cdots\lambda_n^2 q_{k+1}
     = \lambda_2^2\cdots\lambda_n^2 (\lambda_1 q_{k } + \tilde q_{k+1})
    \quad\hbox{for some } \ q_{k }\in P_{k}(T).
  } By \eqref{f22}, we have further, because $\nabla \tilde q_{k+1}\cdot \b n |_{e_1}=0$,
\a{ q=  \lambda_2^2\cdots\lambda_n^2 (\lambda_1^2 q_{k-1 } + \tilde q_{k+1})
    \quad\hbox{for some } \ q_{k-1}\in P_{k-1}(T).  }
 Letting $p=q_{k-1}$ in \eqref{f21},  it follows that, by \eqref{bound},
\a{ \|q_{k-1}\|_T^2 &\le (h_T/h_{e_0})^{2k-2} \|q_{k-1}\|_{T_0} ^2
          \le 64^{2k-2} 2^{2n-2} ( \lambda_2^2\cdots\lambda_n^2  q_{k-1},q_{k-1})_{T_0}\\
    &\le 2^{2n+12k-14} (2h_T/h_{e_0})^{ 2}( \lambda_1^2\cdots\lambda_n^2  q_{k-1},q_{k-1})_{T_{0,0}}
    \\ & \le 2^{2n+12k}(2h_T/h_{e_0})^{ 2n+2k-2} ( \lambda_1^2\cdots\lambda_n^2  q_{k-1},q_{k-1})_{T}
    \\&   =  2^{16n+26k-14} ( \lambda_2^2\cdots\lambda_n^2 \tilde q_{k+1},-q_{k-1})_{T}\\
    &\le  2^{16n+26k-14} \frac{8^{2n-2}} {\sin^{2n-2} \alpha_0}
    \| \tilde q_{k+1}\|_{T}  \|q_{k-1}\|_{T}, }
where $T_{0,0}$ is the top half of $T_0$, and we used the fact
   $\max_T \lambda_1=1$ and the fact that the integrant
   on $T_{0,0}$ is a degree $2n+2k-2$ polynomial.
We estimate
\a{ \|q\|_T^2 &= (\lambda_2^4\cdots\lambda_n^4 (\lambda_1^2 q_{k-1} + \tilde q_{k+1}),
   \lambda_1^2 q_{k-1} + \tilde q_{k+1})_T
            \\& \le  \frac{8^{4n-4}} {\sin^{4n-4} \alpha_0} ( \lambda_1^2 q_{k-1} + \tilde q_{k+1},
    \lambda_1^2 q_{k-1} + \tilde q_{k+1})_T
        \\& \le  \frac{8^{4n-4}} {\sin^{4n-4} \alpha_0} 2
          ( \|\lambda_1^2 q_{k-1}\|_T^2 + \| \tilde q_{k+1}\|_T^2 )
      \\ & \le  \frac{2^{12n-11}} {\sin^{4n-4} \alpha_0}
          ( \| q_{k-1}\|_T^2 + \| \tilde q_{k+1}\|_T^2 ). }
Combining above three bounds,  we get
\an{\label{C2}\ad{ \|q\|_T &\le
    \Big(\frac{2^{12n-11}} {\sin^{4n-4} \alpha_0}\Big)^{\frac 12}
       \Big( ( \frac{2^{16n+26k-14}} {\sin^{2n-2} \alpha_0})^2 + 1\Big)^{\frac 12}   \| \tilde q_k\|_T
            \\& \le
      \Big(\frac{2^{16n+20k+5}} {\sin^{4n-4} \alpha_0}\Big)^{\frac 12}
       \Big( ( \frac{2^{16n+26k-14}} {\sin^{2n-2} \alpha_0})^2 + 1\Big)^{\frac 12} h_T  \|g\|_{e_1}
    =: C h_T  \|g\|_{e_1}.
} } The proof is completed.
\end{proof}

\begin{lemma} There exist two positive constants $C_1$ and $C_2$ such
that for any $v=\{v_0,v_b, v_n\bn_e\}\in V_h$, we have
\begin{equation}\label{happy}
C_1 \|v\|_{2,h}\le \3bar v\3bar \leq C_2 \|v\|_{2,h}.
\end{equation}
\end{lemma}

\begin{proof}
For any $v=\{v_0,v_b,v_n\bn_e\}\in V_h$, it follows from the definition of
weak Laplacian (\ref{wl}) and integration by parts that
\begin{eqnarray}
(\Delta_{w} v, \ \varphi)_T &=& ( v_0, \ \Delta\varphi)_T-\l v_b,\
\nabla\varphi\cdot\bn\r_\pT +\l v_n\bn_e\cdot\bn, \ \varphi\r_\pT\nonumber\\
&=&-(\nabla v_0, \ \nabla\varphi)_T+\l v_0-v_b,\
\nabla\varphi\cdot\bn\r_\pT +\l v_n\bn_e\cdot\bn, \ \varphi\r_\pT\nonumber\\
&=&(\Delta v_0, \ \varphi)_T+\l v_0-v_b,\
\nabla\varphi\cdot\bn\r_\pT +\l (v_n\bn_e-\nabla v_0)\cdot\bn, \ \varphi\r_\pT.\label{n-1}
\end{eqnarray}

By letting $\varphi=\Delta_w v$ in (\ref{n-1}) we arrive at
\begin{eqnarray*}
\|\Delta_{w} v\|^2_T &=&(\Delta v_0, \ \Delta_w v)_T+\l v_0-v_b,\
\nabla(\Delta_w v)\cdot\bn\r_\pT +\l (v_n\bn_e-\nabla v_0)\cdot\bn, \ \Delta_w v\r_\pT
\end{eqnarray*}

From the trace inequality (\ref{trace}) and the inverse inequality
we have
\begin{eqnarray*}
\|\Delta_wv\|^2_T &\le& \|\Delta v_0\|_T \|\Delta_w v\|_T+ \|v_0-v_b\|_\pT \|\nabla(\Delta_w v)\|_\pT
  \\ && \ +\|(v_n\bn_e-\nabla v_0)\cdot\bn\|_\pT \|\Delta_w v\|_\pT\\
&\le& C(\|\Delta v_0\|_T + h_T^{-3/2}\|v_0-v_b\|_\pT
  \\ && \ +h_T^{-1/2}\|(v_n\bn_e-\nabla v_0)\cdot\bn\|_\pT) \|\Delta_w v\|_T,
\end{eqnarray*}
which implies
$$
\|\Delta_w v\|_T \le C \left(\|\Delta v_0\|_T +h_T^{-3/2}\|v_0-v_b\|_\pT+h_T^{-1/2}\|(v_n\bn_e-\nabla v_0)\cdot\bn\|_\pT\right),
$$
and consequently
$$\3bar v\3bar \leq C_2 \|v\|_{2,h}.$$
Next we will prove
$$
\sum_{T\in\T_h} h_T^{-3}\|v_0-v_b\|^2_\pT\le C\3bar v\3bar^2.
$$
It follows from (\ref{n-1}) that for any $\varphi\in P_j(T)$,
\begin{eqnarray}
(\Delta_{w} v, \ \varphi)_T
&=&(\Delta v_0, \ \varphi)_T+\l v_0-v_b,\
\nabla\varphi\cdot\bn\r_\pT\nonumber\\
& +&\l (v_n\bn_e-\nabla v_0)\cdot\bn, \ \varphi\r_\pT.\label{n-11}
\end{eqnarray}
By Lemma \ref{l-m1}, there exist a $\varphi_0$ such that for $e\subset\pT$,
\begin{eqnarray}
&&(\Delta v_0,\varphi_0)_T=0,\;\;\l (v_n\bn_e-\nabla v_0)\cdot\bn, \ \varphi_0\r_\pT=0,\nonumber\\
&& \l v_0-v_b,\nabla\varphi_0\cdot\bn\r_{\pT\setminus e}=0,\;\; \l v_0-v_b,\nabla\varphi_0\cdot\bn\r_\pT=\|v_0-v_b\|_e^2.\label{cond1}
\end{eqnarray}
and
\begin{eqnarray}
\|\varphi_0\|_T\le C h_T^{3/2}\|v_0-v_b\|_e.\label{cond2}
\end{eqnarray}
Letting $\varphi=\varphi_0$ in (\ref{n-11}) yields
\begin{eqnarray}
\|v_0-v_b\|_e^2&=&(\Delta_{w} v, \ \varphi_0)_T\le \|\Delta_{w} v\|_T\|\varphi_0\|_T\le Ch_T^{3/2} \|\Delta_{w} v\|_T\|v_0-v_b\|_e,
\end{eqnarray}
which implies
\begin{equation}\label{n-22}
\sum_{T\in\T_h} h_T^{-3}\|v_0-v_b\|^2_\pT\le C\3bar v\3bar^2.
\end{equation}
Similarly, by Lemma \ref{l-m2}, we can have
\begin{equation}\label{n-33}
\sum_{T\in\T_h} h_T^{-1}\|(\nabla v_0-v_n\bn_e)\cdot\bn\|^2_\pT \le C\3bar v\3bar^2.
\end{equation}
Finally, by letting $\varphi=\Delta_w v$ in (\ref{n-11}) we arrive at
\begin{eqnarray*}
\|\Delta v_0\|^2_T &=&(\Delta v_0, \ \Delta_w v)_T-\l v_0-v_b,\
\nabla(\Delta_w v)\cdot\bn\r_\pT \\
   &&\quad \ -  \l (v_n\bn_e-\nabla v_0)\cdot\bn, \ \Delta_w v\r_\pT.
\end{eqnarray*}
Using the trace inequality (\ref{trace}), the inverse inequality and (\ref{n-22})-(\ref{n-33}), one has
\begin{eqnarray*}
\|\Delta v_0\|^2_T &\le& C \|\Delta_w v\|_T\|\Delta v_0\|_T,
\end{eqnarray*}
which gives
\begin{eqnarray*}
\sum_{T\in\T_h}\|\Delta v_0\|^2_T &\le& C \3bar v\3bar^2.
\end{eqnarray*}
We complete the proof.
\end{proof}

\smallskip

\begin{lemma}
The weak Galerkin finite element scheme (\ref{wg}) has a unique
solution.
\end{lemma}

\begin{proof}
It suffices to show that the solution of (\ref{wg}) is trivial if
$f=g=\phi=0$. Take $v=u_h$  in
(\ref{wg}). It follows that
\[
(\Delta_w u_h,\Delta_w u_h)_{\T_h}=0.
\]
Then the norm equivalence (\ref{happy}) implies $\|u_h\|_{2,h}=0$. Consequently, we have $\Delta u_0=0$,
$u_0=u_b,\ \nabla u_0\cdot\bn_e=u_{n}$ on $\pT$.
Thus $u_0$ is a smooth harmonic function  on $\Omega$.
The boundary condition of $u_b=0$  implies that $u_0\equiv 0$ on
$\Omega$. We have $u_0=0$, then $u_b=u_n=0$, which completes the proof.
\end{proof}

\section{An Error Equation}

Let $e_h=u-u_h$. The goal of this section is to obtain  an error equation that $e_h$ satisfies.

\begin{lemma}
For any $v\in V_h^0$, we have
\begin{eqnarray}
(\Delta_we_h,\Delta_wv)_{\T_h}=\ell_1(u,v)+\ell_2(u,v),\label{ee}
\end{eqnarray}
where
\begin{eqnarray*}
\ell_1(u,v)&=& \langle \nabla(\bbQ_h\Delta u-\Delta u)\cdot\bn, v_0-v_b\rangle_{\pT_h},\\
\ell_2(u,v)&=& \langle \Delta u-\bbQ_h\Delta  u, (\nabla v_0-v_n\bn_e)\cdot\bn\rangle_{\pT}.
\end{eqnarray*}
\end{lemma}

\begin{proof}
For $v=\{v_0,v_b,v_n\bn_e\}\in V_h^0$, testing (\ref{pde}) by  $v_0$  and using the fact that
$\sum_{T\in\T_h}\langle \nabla (\Delta u)\cdot\bn, v_b\rangle_\pT=0$ and $\sum_{T\in\T_h}\langle \Delta u, v_n\bn_e\cdot\bn\rangle_\pT=0$ and integration by parts,  we arrive at
\begin{eqnarray}
(f,v_0)&=&(\Delta^2u, v_0)_{\T_h}\nonumber\\
&=&(\Delta u,\Delta v_0)_{\T_h} -\langle \Delta u,
\nabla v_0\cdot\bn\rangle_{\pT_h} + \langle\nabla(\Delta u)\cdot\bn,
v_0\rangle_{\pT_h}\label{m1}\\
&=&(\Delta u,\Delta v_0)_{\T_h} -\langle \Delta u,
(\nabla v_0-v_n\bn_e)\cdot\bn\rangle_{\pT_h}\nonumber\\
& +& \langle\nabla(\Delta u)\cdot\bn,
v_0-v_b\rangle_{\pT_h}.\nonumber
\end{eqnarray}
Next we investigate the term  $(\Delta u,\Delta v_0)_{\T_h}$ in the above equation. Using (\ref{key}), integration by parts and the definition of weak Laplacian, we have
\begin{eqnarray*}
& & (\Delta u, \Delta v_0)_{\T_h}=(\bbQ_h\Delta u, \Delta v_0)_{\T_h} \\
&=& (v_0, \Delta(\bbQ_h\Delta u))_{\T_h} + \langle \nabla v_0\cdot\bn,\
    \bbQ_h\Delta u\rangle_{\pT_h}-\langle v_0, \nabla(\bbQ_h\Delta u)\cdot\bn \rangle_{\pT_h}\nonumber\\
&=&(\Delta_w v,\ \bbQ_h\Delta u)_{\T_h}-\langle v_0-v_b, \nabla(\bbQ_h\Delta u)\cdot\bn\rangle_{\pT_h}
  \\ &&\quad +\langle (\nabla v_0-v_n\bn_e)\cdot\bn, \bbQ_h\Delta
u\rangle_{\pT_h}\nonumber\\
&=&(\Delta_w u,\ \Delta_w v)_{\T_h}- \langle v_0-v_b, \nabla(\bbQ_h\Delta u)\cdot\bn\rangle_{\pT_T}
  \\ &&\quad +\langle (\nabla v_0-v_n\bn_e)\cdot\bn, \bbQ_h\Delta
u\rangle_{\pT_h}.
\end{eqnarray*}
Combining  the above equation with (\ref{m1}) gives
\begin{eqnarray}
(f,v_0)&=&(\Delta^2u, v_0)_{\T_h}\nonumber\\
&=&(\Delta_w u,\ \Delta_w v)_{\T_h}-\langle v_0-v_b, \nabla(\bbQ_h\Delta u-\Delta u)\cdot\bn\rangle_{\pT_h}\nonumber\\
&-&\langle (\nabla v_0-v_n\bn_e)\cdot\bn, \Delta u-\bbQ_h\Delta u \rangle_{\pT}.\label{mmmm}
\end{eqnarray}
which implies that
\begin{eqnarray*}
(\Delta_w u,\ \Delta_w v)_{\T_h}=(f,v_0)+\ell_1(u,v)+\ell_2(u,v).
\end{eqnarray*}
The error equation follows from subtracting (\ref{wg}) from the above equation,
\begin{eqnarray*}
(\Delta_w e_h,\ \Delta_w v)_{\T_h}=\ell_1(u,v)+\ell_2(u,v).
\end{eqnarray*}
We have proved the lemma.
\end{proof}

\section{An Error Estimate in $H^2$}

We will obtain the optimal convergence rate for the solution $u_h$ of the stabilizer free WG method in (\ref{wg}) in a discrete $H^2$ norm.

\begin{lemma}\label{l2}
Let $k\ge 2$ and $w\in H^{\max\{k+1,4\}}(\Omega)$. There exists a constant $C$ such that the following estimates hold true:
\begin{eqnarray}
&  \left(\sum_{T\in\T_h} h_T\|\Delta w-\bbQ_h\Delta w\|_{\partial
T}^2\right)^{\frac12}
\leq C h^{k-1}\|w\|_{k+1},\label{mmm1}\\
&  \left(\sum_{T\in\T_h} h_T^3\|\nabla(\Delta w-\bbQ_h\Delta
w)\|_\pT^2\right)^{\frac12} \leq Ch^{k-1}(\|w\|_{k+1}
+h\delta_{k,2}\|w\|_4).
\label{mmm2}
\end{eqnarray}
Here $\delta_{i,j}$ is the usual Kronecker's delta with value $1$
when $i=j$ and value $0$ otherwise.
\end{lemma}

The above lemma can be proved by using the trace inequality (\ref{trace}) and the definition of $\bbQ_h$. The proof can also be found in \cite{wg-bi1}.

\begin{lemma} Let $w\in H^{\max\{k+1,4\}}(\Omega)$ for $k\ge 2$ and $v\in V_h$.
There exists a constant $C$ such that
\begin{eqnarray}
|\ell_1(w, v)|&\le&  C h^{k-1}\left(\|w\|_{k+1} + h\delta_{k,2} \|w\|_4\right)\3bar v\3bar.\label{mm1}\\
|\ell_2(w, v)|&\le& Ch^{k-1}\|w\|_{k+1}\3bar v\3bar.\label{mm2}
\end{eqnarray}
\end{lemma}

\begin{proof}
Using the Cauchy-Schwartz inequality, (\ref{mmm1})-(\ref{mmm2}) and (\ref{happy}), we have
\begin{eqnarray}
\ell_1(w,v)&=&\left|\sum_{T\in\T_h}\langle \nabla(\Delta w-\bbQ_h\Delta
w)\cdot\bn, v_0-v_b\rangle_\pT\right| \nonumber\\
\le && \left(\sum_{T\in\T_h}h_T^3\|\nabla(\Delta w-\bbQ_h\Delta
w)\|_\pT^2\right)^{\frac12}
\left(\sum_{T\in\T_h}h_T^{-3}\|v_0-v_b\|^2_{\pT}\right)^{\frac12}\nonumber\\
\le && C h^{k-1}(\|w\|_{k+1} + h\delta_{k,2} \|w\|_4) \3barv\3bar,\label{ell-1}
\end{eqnarray}
and
\begin{eqnarray}
\ell_2(w,v)&=&\left|\sum_{T\in\T_h} \langle \Delta w-\bbQ_h\Delta w, (\nabla
v_0-v_n\bn_e)\cdot\bn\rangle_\pT\right|\nonumber\\
\le &&\left(\sum_{T\in\T_h} h_T\|\Delta w-\bbQ_h\Delta
w\|_\pT^2\right)^{\frac12} \left(\sum_{T\in\T_h} h_T^{-1}
\|(\nabla v_0-v_n\bn_e)\cdot\bn\|_\pT^2\right)^{\frac12}\nonumber\\
\le && C h^{k-1}\|w\|_{k+1} \3bar v\3bar.\label{ell-2}
\end{eqnarray}
We have completed the proof.
\end{proof}

\begin{lemma}
Let  $w\in H^{\max\{k+1,4\}}(\Omega)$,  then
\begin{equation}\label{eee2}
\3bar w-Q_hw\3bar\le Ch^{k-1}\|w\|_{k+1}.
\end{equation}
\end{lemma}
\begin{proof}
For any $T\in\T_h$, it follows from (\ref{wl}), integration  by parts, (\ref{trace}) and inverse inequality,
\begin{eqnarray*} &&
\|\Delta_w(w-Q_hw)\|_T^2 \\
 &=&(\Delta_w(w-Q_hw), \Delta_w(w-Q_hw))_{T}\\
 &=&(w-Q_0w, \Delta (\Delta_w(w-Q_hw)))_{T}-\l w-Q_bw, \nabla (\Delta_w(w-Q_hw))\cdot\bn\r_\pT\\
& &\ + \l (\nabla w\cdot\bn_e-Q_n(\nabla w\cdot\bn_e)\cdot\bn, \Delta_w(w-Q_hw)\r_{\pT}\\
&\le&C(h_T^{-2}\|w-Q_0w\|_T+h_T^{-3/2}\|w-Q_bw\|_\pT\\
& &\ + h_T^{-1/2}\|\nabla w\cdot\bn_e-Q_n(\nabla w\cdot\bn_e)\|_\pT)\|\Delta_w(w-Q_hw)\|_T\\
&\le& Ch^{k-1}|w|_{k+1, T}\|\Delta_w(w-Q_hw)\|_T.
\end{eqnarray*}
Using the above inequality and taking the summation of it over $T$,  we derive (\ref{eee2}) and prove the lemma.
\end{proof}

\begin{theorem} Let $u_h\in V_h$  be the weak Galerkin finite element solution arising from
(\ref{wg}). Assume that the exact solution $u\in H^{\max\{k+1,4\}}(\Omega)$. Then, there
exists a constant $C$ such that
\begin{equation}\label{err1}
\3bar u-u_h\3bar \le Ch^{k-1}\left(\|u\|_{k+1}+h\delta_{k,2}\|u\|_{4}\right).
\end{equation}
\end{theorem}
\begin{proof}
Let $\epsilon_h=Q_hu-u_h\in V_h^0$. It is straightforward to obtain
\begin{eqnarray}
\3bar e_h\3bar^2&=&(\Delta_we_h, \Delta_we_h)_{\T_h}\label{eee1}\\
&=&(\Delta_we_h, \Delta_w(u-u_h))_{\T_h}\nonumber\\
&=&(\Delta_we_h,\Delta_w(Q_hu-u_h))_{\T_h}+(\Delta_we_h, \Delta_w(u-Q_hu))_{\T_h}\nonumber\\
&=&(\Delta_we_h,\Delta_w\epsilon_h)_{\T_h}+(\Delta_we_h, \Delta_w(u-Q_hu))_{\T_h}.\nonumber
\end{eqnarray}
Next, we bound the two terms on the right hand side  in (\ref{eee1}).
Letting $v=\epsilon_h\in V_h^0$ in (\ref{ee})  and using (\ref{mm1})-(\ref{mm2}) and (\ref{eee2}), we have
\begin{eqnarray}
|(\Delta_we_h,\Delta_w\epsilon_h)_{\T_h}|&\le&|\ell_1(u,\epsilon_h)|+|\ell_2(u,\epsilon_h)|\nonumber\\
&\le& Ch^{k-1}(\|u\|_{k+1}+h\delta_{k,2}\|u\|_{4})\3bar \epsilon_h\3bar\nonumber\\
&\le& Ch^{k-1}(\|u\|_{k+1}+h\delta_{k,2}\|u\|_{4})(\3bar u-Q_hu\3bar+\3bar u-u_h\3bar)\nonumber\\
&\le& Ch^{2(k-1)}(\|u\|^2_{k+1}+h^2\delta_{k,2}\|u\|^2_{4})+\frac14 \3bare_h\3bar^2.\label{eee3}
\end{eqnarray}
The estimate (\ref{eee2}) implies
\begin{eqnarray}
|(\Delta_we_h, \Delta_w(u-Q_hu))_{\T_h}|&\le& C\3bar u-Q_hu\3bar \3bar e_h\3bar\nonumber\\
&\le& Ch^{2(k-1)}\|u\|^2_{k+1}+\frac14\3bar e_h\3bar^2.\label{eee4}
\end{eqnarray}
Combining the estimates (\ref{eee3}) and  (\ref{eee4}) with (\ref{eee1}), we arrive
\[
\3bar e_h\3bar \le Ch^{k-1}\left(\|u\|_{k+1}+h\delta_{k,2}\|u\|_{4}\right),
\]
which completes the proof.
\end{proof}

\section{Error Estimates in $L^2$ Norm}

In this section, we will provide an estimate for the
$L^2$ norm of the WG solution $u_h$.

Recall that $e_h=u-u_h$ and $\epsilon_h=Q_hu-u_h=\{\epsilon_0, \epsilon_b,\epsilon_n\bn_e\}\in V_h^0$.

Let us consider the following dual problem
\begin{eqnarray}
\Delta^2w&=& \epsilon_0\quad
\mbox{in}\;\Omega,\label{dual}\\
w&=&0\quad\mbox{on}\;\partial\Omega,\label{dual1}\\
\nabla w\cdot\bn&=&0\quad\mbox{on}\;\partial\Omega.\label{dual2}
\end{eqnarray}
The $H^{4}$ regularity assumption of the dual problem implies the
existence of a constant $C$ such that
\begin{equation}\label{reg}
\|w\|_4\le C\|\epsilon_0\|.
\end{equation}

\begin{theorem}
Let $u_h\in V_h$ be the weak Galerkin finite element solution
arising from (\ref{wg}). Assume that the exact solution $u\in H^{k+1}(\Omega)$ and (\ref{reg}) holds true.
 Then, there exists a constant $C$ such that
\begin{equation}\label{err2}
\|Q_0u-u_0\| \le Ch^{k+1-\delta_{k,2}}(\|u\|_{k+1}+ h\delta_{k,2}\|u\|_4).
\end{equation}
\end{theorem}

\begin{proof}
Testing (\ref{dual}) by $\epsilon_0$  and then using the equation (\ref{mmmm}) with $u=w$ and $v=\epsilon_h$, we obtain
\begin{eqnarray*}
\|\epsilon_0\|^2&=&(\Delta^2 w, \epsilon_0)\\
&=&(\Delta_w w,\ \Delta_w \epsilon_h)_{\T_h}-\langle \epsilon_0-\epsilon_b, \nabla(\bbQ_h\Delta w-\Delta w)\cdot\bn\rangle_{\pT_h}\nonumber\\
&&\ - \langle (\nabla \epsilon_0-\epsilon_n\bn_e)\cdot\bn, \Delta w-\bbQ_h\Delta w \rangle_{\pT}\\
&=&(\Delta_w w,\ \Delta_w \epsilon_h)_{\T_h}-\ell_1(w,\epsilon_h)-\ell_2(w,\epsilon_h)
\end{eqnarray*}
The error equation (\ref{ee}) gives
\begin{eqnarray*}
(\Delta_w w,\ \Delta_w \epsilon_h)_{\T_h}&=&(\Delta_w w,\ \Delta_w e_h)_{\T_h}+(\Delta_w w,\ \Delta_w (Q_hu-u))_{\T_h}\\
&=&(\Delta_w e_h,\ \Delta_w Q_hw)_{\T_h}+(\Delta_w e_h,\ \Delta_w (w-Q_hw))_{\T_h}\\
  && \ +(\Delta_w w,\ \Delta_w (Q_hu-u))_{\T_h}\\
&=&\ell_1(u,Q_hw)+\ell_2(u, Q_hw)+(\Delta_w e_h,\ \Delta_w (w-Q_hw))_{\T_h}\\
  && \ +(\Delta_w w,\ \Delta_w (Q_hu-u))_{\T_h}.
\end{eqnarray*}
Combining the two equations above, we obtain
\begin{eqnarray*}
  \|\epsilon_0\|^2&=&\ell_1(u,Q_hw)+\ell_2(u, Q_hw)+(\Delta_w e_h,\ \Delta_w (w-Q_hw))_{\T_h}\\
  && \ +(\Delta_w w,\ \Delta_w (Q_hu-u))_{\T_h}
   -\ell_1(w,\epsilon_h)-\ell_2(w,\epsilon_h)\\
&=&I_1+I_2+I_3+I_4+I_5+I_6.
\end{eqnarray*}
Next, we will estimate the all the terms on the right hand side of the above equation.
Using the Cauchy-Schwartz inequality, (\ref{trace}) and (\ref{mmm2}), we have
\begin{eqnarray*}
I_1&=&|\ell_1(u,Q_hw)|=\left|\sum_{T\in\T_h}\langle \nabla(\Delta u-\bbQ_h\Delta u)\cdot\bn, Q_0w-Q_bw\rangle_\pT\right| \\
&\le  & \left(\sum_{T\in\T_h}h^3_T\|\nabla(\Delta u-\bbQ_h\Delta
u)\|_\pT^2\right)^{\frac12}
\left(\sum_{T\in\T_h}h_T^{-3}\|Q_0w-Q_bw\|^2_{\pT}\right)^{\frac12}\\
&\le  & \left(\sum_{T\in\T_h}h^3_T\|\nabla(\Delta u-\bbQ_h\Delta
u)\|_\pT^2\right)^{\frac12}
\left(\sum_{T\in\T_h}h_T^{-3}\|Q_0w-w\|^2_{\pT}\right)^{\frac12}\\
&\le  & C h^{k+1-\delta_{k,2}}\left(\|u\|_{k+1}+h\delta_{k,2}\|u\|_{4}\right)\|w\|_4,
\end{eqnarray*}
Similarly, by the Cauchy-Schwartz inequality, (\ref{mmm1}) and (\ref{trace}), we have
\begin{eqnarray*}
I_2&=&|\ell_2(u,Q_hw)|=\left|\sum_{T\in\T_h} \langle \Delta u-\bbQ_h\Delta u, (\nabla
Q_0w\cdot\bn-Q_n(\nabla w\cdot\bn)\rangle_\pT\right|\nonumber\\
&\le&\left(\sum_{T\in\T_h} h_T\|\Delta u-\bbQ_h\Delta u\|_\pT^2\right)^{\frac12}\times\\
& & \left(\sum_{T\in\T_h} h^{-1} _T
(\|\nabla Q_0w\cdot\bn-\nabla w\cdot\bn\|_\pT^2+\|\nabla w\cdot\bn-Q_n(\nabla w\cdot\bn)\|_\pT^2)\right)^{\frac12}\nonumber\\
&\le& C h^{k+1-\delta_{k,2}}\|u\|_{k+1} \|w\|_4.
\end{eqnarray*}
It follows from (\ref{err1}) and (\ref{eee2}),
\begin{eqnarray*}
I_3&=&(\Delta_w e_h,\ \Delta_w (w-Q_hw))_{\T_h}\le  C h^{k+1-\delta_{k,2}}\|u\|_{k+1} \|w\|_4.
\end{eqnarray*}

To bound $I_4$, we define a $L^2$ projection element-wise onto $P_1(T)$ denoted by $R_h$. Then it follows from the definition of weak Laplacian (\ref{wl})
\begin{eqnarray*} &&
(\Delta_w (Q_hu-u),\ R_h\Delta_w w)_{T}\\
 &=&( Q_0u-u, \ \Delta( R_h\Delta_w w))_T-\l Q_bu-u,\
\nabla( R_h\Delta_w w)\cdot\bn\r_\pT \\
& &\ + \l (Q_n(\nabla u\cdot\bn_e)-\nabla u\cdot\bn_e)\cdot\bn, \  R_h\Delta _w w\r_\pT=0.
\end{eqnarray*}
Using the equation above and (\ref{eee2}) and the definition of $R_h$, we have
\begin{eqnarray*}
I_4&=&|\Delta_w (Q_hu-u),\ \Delta_w w)_{\T_h}|\\
&=&|(\Delta_w (Q_hu-u),\ \Delta_w w-R_h\Delta_w w)_{\T_h}|\\
&\le& Ch^{k+1}\|u\|_{k+1}\|w\|_4.
\end{eqnarray*}
Using (\ref{m1}), (\ref{mm1}), (\ref{err1}) and (\ref{eee2}), we have
\begin{eqnarray*}
I_5&=&|\ell_1(w,\epsilon_h)|\le C h^{2-\delta_{k,2}} \|w\|_{4}\3bar\epsilon_h\3bar\le C h^{2-\delta_{k,2}}\|w\|_{4}(\3bar Q_hu-u\3bar + \3bar e_h\3bar) \\
&\le& Ch^{k+1-\delta_{k,2}}\|u\|_ {k+1}\|w\|_4.
\end{eqnarray*}
Similarly, we obtain
\begin{eqnarray*}
I_6=|\ell_2(w,\epsilon_h)|\le Ch^{k+1-\delta_{k,2}}\|u\|_ {k+1}\|w\|_4.
\end{eqnarray*}
Combining all the estimates above yields
$$
\|\epsilon_0\|^2 \leq C h^{k+1-\delta_{k,2}}(\|u\|_{k+1}+ h\delta_{k,2}\|u\|_4) \|w\|_4.
$$
It follows from the above inequality and
the regularity assumption (\ref{reg}).
 $$
\|\epsilon_0\|\leq C h^{k+1-\delta_{k,2}}(\|u\|_{k+1}+ h\delta_{k,2}\|u\|_4).
$$
We have completed the proof.
\end{proof}

\section{Numerical Test} \label{Section:numerical-results}

 We solve the following 2D biharmonic equation on the unit square:
\begin{align} \label{s1}  \Delta^2 u =f  ,  \quad (x,y)\in\Omega=(0,1)^2,
\end{align} with the boundary conditions $u=g_1$ and $\nabla u\cdot \b n=g_2$ on $\partial \Omega$.
Here $f$, $g_1$ and $g_2$ are chosen so that the exact solution is
\a{ u=e^{x+y}. }

\begin{figure}[h!]
 \begin{center} \setlength\unitlength{1.25pt}
\begin{picture}(260,80)(0,0)
  \def\tr{\begin{picture}(20,20)(0,0)\put(0,0){\line(1,0){20}}\put(0,20){\line(1,0){20}}
          \put(0,0){\line(0,1){20}} \put(20,0){\line(0,1){20}}  \put(20,0){\line(-1,1){20}}\end{picture}}
 {\setlength\unitlength{5pt}
 \multiput(0,0)(20,0){1}{\multiput(0,0)(0,20){1}{\tr}}}

  {\setlength\unitlength{2.5pt}
 \multiput(45,0)(20,0){2}{\multiput(0,0)(0,20){2}{\tr}}}

  \multiput(180,0)(20,0){4}{\multiput(0,0)(0,20){4}{\tr}}

 \end{picture}\end{center}
\caption{\label{grid1} The first three levels of grids used in the computation of Table \ref{t1}. }
\end{figure}

In the first computation, the level one grid consists of two unit right triangles
     cutting from the unit square by a forward
  slash.   The high level grids are the half-size refinements of the previous grid.
The first three levels of grids are plotted in Figure \ref{grid1}.
The error and the order of convergence for the method are shown in Tables \ref{t1}.
Here on triangular grids,  we let $j=k+2$ defined in (\ref{wl}) for computing the  weak
  Laplacian $\Delta_w v$.
The numerical results confirm the convergence theory.

\begin{table}[h!]
  \centering \renewcommand{\arraystretch}{1.1}
  \caption{Error profiles and convergence rates for \eqref{s1} on triangular grids (Figure \ref{grid1}) }\label{t1}
\begin{tabular}{c|cc|cc|cc}
\hline
level & $\|u_h-  u\|_0 $  &rate & $ |u_h-u|_{1,h} $ &rate  & $\3bar u_h- u\3bar $ &rate    \\
\hline
 &\multicolumn{6}{c}{by the $P_2$ weak Galerkin finite element } \\ \hline
 5&   0.7913E-04 & 1.96&   0.5596E-03 & 2.00&   0.2764E+00 & 1.00\\
 6&   0.2016E-04 & 1.97&   0.1412E-03 & 1.99&   0.1383E+00 & 1.00\\
 7&   0.5049E-05 & 2.00&   0.3547E-04 & 1.99&   0.6912E-01 & 1.00\\
 \hline
 &\multicolumn{6}{c}{by the $P_3$ weak Galerkin finite element } \\ \hline
  3&   0.3788E-05 & 4.20&   0.1398E-03 & 3.09&   0.2949E-01 & 2.00\\
 4&   0.2114E-06 & 4.16&   0.1713E-04 & 3.03&   0.7384E-02 & 2.00\\
 5&   0.1284E-07 & 4.04&   0.2128E-05 & 3.01&   0.1848E-02 & 2.00\\
 \hline
\end{tabular}%
\end{table}%

In the next computation,  we use a family of polygonal grids (with pentagons)
   shown in Figure \ref{5g}.
We let the polynomial degree $j=k+3$ for the weak Laplacian on such polygonal meshes.
The rate of convergence is listed in Table \ref{t2}.
The convergence history confirms the theory.

\begin{figure}[h!]
 \begin{center} \setlength\unitlength{1.45pt}
\begin{picture}(245,80)(0,0)
  \def\tr{\begin{picture}(20,20)(0,0)
            \put(0,0){\line(1,0){20}}\put(0,20){\line(1,0){20}}
          \put(0,0){\line(0,1){20}} \put(20,0){\line(0,1){20}}
          \put(10,0){\line(0,1){3}}\put(10,20){\line(0,-1){3}}
           \put(0,10){\line(1,0){3}}\put(20,10){\line(-1,0){3}}
           \put(3,10){\line(1,1){7}}\put(3,10){\line(1,-1){7}}
         \put(17,10){\line(-1,-1){7}}\put(17,10){\line(-1,1){7}}
  \end{picture}}
 {\setlength\unitlength{5.8pt}
 \multiput(0,0)(20,0){1}{\multiput(0,0)(0,20){1}{\tr}}}

  {\setlength\unitlength{2.9pt}
 \multiput(41,0)(20,0){2}{\multiput(0,0)(0,20){2}{\tr}}}

  \multiput(164,0)(20,0){4}{\multiput(0,0)(0,20){4}{\tr}}

 \end{picture}\end{center}
\caption{  The first three levels of polygonal grids used in the computation of Table \ref{t2}. }
   \label{5g}
\end{figure}

\begin{table}[h!]
  \centering \renewcommand{\arraystretch}{1.1}
  \caption{Error profiles and convergence rates for \eqref{s1}
        on polygonal grids (Figure \ref{5g}) }\label{t2}
\begin{tabular}{c|cc|cc|cc}
\hline
level & $\|u_h-  u\|_0 $  &rate & $ |u_h-u|_{1,h} $ &rate  & $\3bar u_h- u\3bar $ &rate    \\
\hline
 &\multicolumn{6}{c}{by the $P_2$ weak Galerkin finite element } \\ \hline
 3&   0.5699E-03 &  2.6&   0.8766E-02 &  1.9&   0.4895E+01 &  1.0 \\
 4&   0.1035E-03 &  2.5&   0.2346E-02 &  1.9&   0.2445E+01 &  1.0 \\
 5&   0.2477E-04 &  2.1&   0.6175E-03 &  1.9&   0.1222E+01 &  1.0 \\
 6&   0.6835E-05 &  1.9&   0.1598E-03 &  2.0&   0.6112E+00 &  1.0 \\ \hline
 &\multicolumn{6}{c}{by the $P_3$ weak Galerkin finite element } \\ \hline
 1&   0.1571E-02 &  0.0&   0.1905E-01 &  0.0&   0.3251E+01 &  0.0 \\
 2&   0.9077E-04 &  4.1&   0.2259E-02 &  3.1&   0.7397E+00 &  2.1 \\
 3&   0.5368E-05 &  4.1&   0.2888E-03 &  3.0&   0.1793E+00 &  2.0 \\
 4&   0.3474E-06 &  3.9&   0.3939E-04 &  2.9&   0.4445E-01 &  2.0 \\
 \hline
\end{tabular}%
\end{table}%

\end{document}